\documentclass[12pt]{amsart}
\usepackage{amssymb}
\usepackage{amsmath,amscd}
\usepackage{amsfonts,latexsym,verbatim,amscd,mathrsfs,color,array}
\usepackage[all,cmtip]{xy}
\usepackage{mathrsfs}
\usepackage{multirow}
\usepackage{graphicx}
\usepackage{amsfonts, amsthm}
\usepackage{bbm}
\usepackage[hmargin=1in,vmargin=1in]{geometry}
\usepackage{mathdots}
\usepackage{stmaryrd}
\usepackage{MnSymbol}
\usepackage{bbm}
\usepackage[hidelinks]{hyperref}
\usepackage{enumitem}
\usepackage[all]{xy}
\usepackage{tikz-cd}

\allowdisplaybreaks

\def\XXint#1#2#3{{\setbox0=\hbox{$#1{#2#3}{\int}$ }
		\vcenter{\hbox{$#2#3$ }}\kern-.6\wd0}}

\newcommand{\Tr}{\operatorname{Tr}}
\newcommand{\transint}{\cap\kern-0.63em|\kern0.7em}

\DeclareMathSymbol{\intprod}{\mathbin}{MnSyC}{'270}

\newcommand{\LB}{\left[}

\newcommand{\p}{{ \partial}}

\newcommand{\Z}{{\mathbb Z}}
\newcommand{\N}{{\mathbb N}}

\newcommand{\R}{{\mathbb R}}

\newcommand{\SU}{{\mathrm{SU} }}

\newcommand{\SO}{{\mathrm{SO} }}
\newcommand{\YM}{{\mathcal{YM} }}

\newcommand{\dist}{{\operatorname{dist}}}

\renewcommand{\p}{{\partial}}
\newcommand{\eps}{{\varepsilon}}

\newtheorem{thm}{Theorem}[section]
\newtheorem{lemma}[thm]{Lemma}
\newtheorem*{lemma*}{Lemma}

\newtheorem{cor}[thm]{Corollary}

\newtheorem*{conj*}{Conjecture}

   \newtheoremstyle{others}% name
     {3pt}%      Space above
     {2pt}%      Space below
     {}%         Body font
     {}%         Indent amount (empty = no indent, \parindent = para indent)
     {\bf}% Thm head font
     {.}%        Punctuation after thm head
     {.5em}%     Space after thm head: " " = normal interword space;
     {}%         Thm head spec (can be left empty, meaning `normal')

\theoremstyle{others}
\newtheorem{rmk}[thm]{Remark}
\newtheorem*{rmk*}{Remark}
\newtheorem{defn}[thm]{Definition}

\newtheorem*{question*}{Question}

\numberwithin{equation}{section}

\usepackage{mathtools} 
\usepackage{faktor}
\usepackage{dynkin-diagrams}
\usepackage{mathrsfs}
\usepackage{hyperref}
\usepackage{tikz-cd}
\usepackage{bbm}

\usepackage{romanbar}

\begin{document}

\title[Convergence of Yang-Mills flow]{Convergence of the Yang-Mills flow on ALE gravitational instantons}
\author{Anuk Dayaprema}
\address{University of Wisconsin, Madison}
\email{dayaprema@wisc.edu}
\author{Alex Waldron}
\address{University of Wisconsin, Madison}
\email{waldron@math.wisc.edu}
\date{\today}

\begin{abstract}
    We prove a sharp convergence theorem for the Yang-Mills flow on an $\SU(r)$-bundle over a locally hyperK\"ahler ALE 4-manifold. % under a sharp hypothesis on the self-dual curvature of the initial connection.
    Our main result is a noncompact version of the ``parabolic gap theorem'' \cite[Theorem 1.1]{gappaper} previously established by the authors.
\end{abstract}

\maketitle

\tableofcontents

\thispagestyle{empty}

\section{Introduction and statement of main result}

 A time-dependent family of connections $A(t)$ on a vector bundle $E$ over a Riemannian manifold $(M,g)$ is said to evolve by \emph{Yang-Mills flow} if
\begin{equation}\label{ymf}
\frac{\p A}{\p t} = - D_{A(t)}^*F_{A(t)}.
\end{equation}
Here $F_{A(t)}$ is the curvature of $A(t)$ and $D_{A(t)}^*$ is the formal adjoint of the covariant derivative. 
The semi-parabolic equation (\ref{ymf}) is the downward gradient-flow of the \emph{Yang-Mills functional}
$$\mathcal{YM}(A) = \frac12 \int_{M} |F_A|_g^2 \, dV_g,$$
which originated in physics and has become a fundamental object of study in differential geometry.

The goal of this paper is to extend the authors' recent ``parabolic gap theorem'' \cite[Theorem 1.1]{gappaper}, valid over the 4-sphere, to a certain well-known family of noncompact spaces. Recall that an \emph{elliptic} gap theorem states that any Yang-Mills connection with energy below a certain value is an instanton, i.e., has self-dual or anti-self-dual curvature and is a minimizer of $\YM(\cdot)$ on the given bundle (see \cite[\S 1]{gappaper} for a brief survey). A \emph{parabolic} gap theorem states that any initial connection with energy below a certain value is deformed smoothly by (\ref{ymf}) to an instanton. The latter is a stronger and more natural statement in that it applies to general connections rather than only to solutions of the Yang-Mills equation. % \emph{a priori.} %This stronger statement can potentially be used to study the topology of instanton moduli spaces; see 

\begin{comment}
    The well-known noncompact spaces in question are ALE hyperK\"ahler 4-manifolds, a.k.a. \emph{gravitational instantons}. These are complete Ricci-flat 4-manifolds with vanishing self-dual Weyl tensor and Euclidean volume growth at infinity, having a single end modeled on the space $\left( \R^4 \setminus \{0\} \right) / \Gamma,$ where $\Gamma \subset \SU(2)$ is a finite subgroup. Kronheimer \cite{kronheimerconstruction} was able to classify these spaces completely, identifying each one with a canonical hyperK\"ahler structure on a resolution of the orbifold $\R^4 / \Gamma.$
\end{comment}
The well-known noncompact spaces in question are ALE locally hyperK\"ahler 4-manifolds. These are complete, Ricci-flat 4-manifolds with vanishing self-dual Weyl tensor and maximal volume growth at infinity, having a single end modeled on the space $\left( \R^4 \setminus \{0\} \right) / \Gamma,$ where $\Gamma \subset \SO(4)$ is a finite subgroup. In the simply-connected case, 
Kronheimer \cite{kronheimerconstruction, kronheimertorelli} gave a complete construction of all such $M,$ identifying each one with a canonical hyperK\"ahler structure on a resolution of the orbifold $\R^4 / \Gamma.$ Wright \cite{wrightquotientsofgravitons} later completed the classification by determining the possible finite quotients of the Kronheimer examples.
%Kronheimer \cite{kronheimerconstruction} was able to classify these spaces completely, identifying each one with a canonical hyperK\"ahler structure on a resolution of the orbifold $\R^4 / \Gamma.$ %For this reason we will alternately refer to an ALE hyperK\"ahler 4-manifold as a ``Kronheimer ALE space.''
Note that a classification of complete 4D hyperK\"ahler manifolds with less-than-maximal volume growth has also been achieved recently \cite{chenchengravitationalI,chenchengravitationalII,chenchengravitationalIII,sunzhangcollapsing}, but in the present work we focus on the ALE case.

%\subsection{Motivation}
We have multiple motivations for proving parabolic gap theorems on gravitational instantons.

\begin{enumerate}[label=\arabic*., wide, labelwidth=!, labelindent=15pt]
    \item 
    %Even in the case of Euclidean 4-space, 
    The convergence problem for (\ref{ymf}) on a complete noncompact space presents new analytic challenges as compared to the compact case. Indeed, the work of Gustaffson, Nakanishi, and Tsai \cite{gustafsonnakanishitsai} (see also Sire, Wei, Zheng, and Zhou \cite{sireweiglobalosc}) demonstrates that, strictly speaking, the parabolic gap \emph{fails} on $\R^4$; there exist smooth $\SO(4)$-equivariant initial connections $H^1$-close to the standard instanton that fail to converge at infinite time under (\ref{ymf}). We will see that this problem can be avoided by imposing a mild additional decay hypothesis on the self-dual curvature of the initial connection.

\item 
A parabolic gap theorem on $\R^4$ has already been established by Oh and Tataru \cite[Theorem 2.11]{ohtatarucaloric} for connections with Yang-Mills energy less than $8\pi^2$ and zero topological charge (see Definition \ref{defn:topologicalcharge}).
Their result is used to construct a ``caloric gauge,'' an essential step in the proof of the Threshold Conjecture for the hyperbolic Yang-Mills equations \cite{ohtatarucaloric, ohtataruhyperboliccaloric, ohtataruarbitrarytopologicalclass, ohtataruthreshold}. Our main result generalizes that of Oh and Tataru to higher topological charge. %, by requiring a smallness assumption only on the self-dual part of the curvature.
In particular, our conclusion is that the flow (\ref{ymf}) converges smoothly to an ASD instanton of the same charge as that of the initial connection, rather than to the flat connection. We expect that this will open the way to generalizations of Oh-Tataru's work on the Threshold Conjecture. % on the hyperbolic Yang-Mills equation.

\item
The parabolic gap question for (\ref{ymf}) on a bundle over a gravitational instanton is closely analogous to the question of stability of the underlying space under Ricci flow, studied by \cite{DeruelleKroncke20,DeruelleOzuch20,KronckePetersen20,stolarskiwaldron}. In the case of (\ref{ymf}), we are able to obtain convergence under a vastly weaker hypothesis on the initial data. In particular, we do not need to assume closeness to any ASD instanton. %In the future it may be possible to obtain stability results for the Ricci flow on an ALE hyperK\"ahler background under similarly weakened hypotheses.

\item
Last, we expect that our results will be useful in the future for studying the topology of moduli spaces of ASD instantons on ALE gravitational instantons \cite{kronheimernakajima, nakajima90}. %in the manner of Taubes's work on ADHM moduli spaces (see \cite{taubespathconnected} and \cite[\S 3.3]{parabolicgap}). %While the celebrated work of Nakajima \cite{} and Kronheimer-Nakajima \cite{} constructs all ASD instantons on ALE instantons in terms of matrix data, the topology of these moduli spaces is not fully understood.

\end{enumerate}

%\item We shall see that the analytic requirements for our proof are provided by the geometry of the gravitational instanton in a rather beautiful way.

%\end{itemize}

%\subsection{Statement of results}
The following summarizes our main results:

\begin{thm}\label{thm:main}
    Let $(M, g)$ be an ALE gravitational instanton asymptotic to $\left( \R^4 \setminus \{0\} \right) / \Gamma,$ %with group $\Gamma$ at infinity {\color{red} fix terminology},
    oriented such that the self-dual Weyl curvature vanishes. Let $E \to M$ be an $\SU(r)$-bundle and $A_0$ a connection on $E$ with $\YM(A_0) < \infty$ and
        \begin{equation}\label{main:gap}
        \|F^+_{A_0} \|_{L^2(M)}^2 < \frac{8 \pi^2}{|\Gamma|}.
    \end{equation}
    Let $A(t),$ $t \in \LB 0, \infty \right),$ be any finite-energy classical solution of the Yang-Mills flow with $A(0) = A_0.$ Denote the self-dual part of the curvature of $A(t)$ by $F^+(t).$ %uniformly bounded energy. % such that $\mathcal{YM}(A(t))$ is locally bounded in $t.$

\vspace{2mm}

\noindent    (a) We have
    \begin{align*}
        \|F^+(t)\|_{L^\infty(M)} = o \left( \frac{1}{t} \right)
    \end{align*}
    and
        \begin{align*}
        \|F^+(t)\|_{L^2(M)} \to 0
    \end{align*}
        as $t \to \infty.$
    %Also, $\mathcal{YM}(A(t))$ is a non-increasing function of $t.$

    \vspace{2mm}

\noindent    (b) Suppose further that $|F^+_{A_0}| \in L^p(M)$ for some $1 < p < 2.$ 
We then have %$|F^+_{A(t)}| \in L^p(M)$ for all $t,$
    \begin{align*}
        \|F^+(t)\|_{L^\infty(M)} = o \left( \frac{1}{t^{\frac{2}{p}}} \right),
    \end{align*}
    and
    \begin{align*}
        \|F^+(t)\|_{L^p(M)} \to 0 %= O\left(\frac{1}{t^{\frac{2}{p}-1}}\right)
    \end{align*}
    as $t \to \infty.$

    \vspace{2mm}

    \noindent (c) Under the above assumptions,  %as $t \to \infty,$ $|F^+(t)|$ converges to zero in $L^p(M)$ and 
$A(t)$ converges in $C^\infty_{loc} \cap L^q$ for each $q > \frac{2p}{2-p}$ to an ASD instanton $A_\infty$ on $E$ with $\kappa(A_\infty) = \kappa(A_0).$ %No discussion of framing at infinity etc.?
\end{thm}
%We will give more precise statements of each of these estimates below (see \ref{}, \ref{}, and \ref{}).

Note that %finite-energy solutions of (\ref{ymf}) do not encounter finite-time singularities---
in view of \cite{lte} or \cite[Theorem 1.1]{gappaper}, assuming a reasonable short-time existence theory on the given 4-manifold, we can always solve (\ref{ymf}) classically on $\LB 0, \infty \right).$ %with any finite-energy initial connection $A_0.$ 
%Solutions can therefore be extended for all time assuming a reasonable short-time existence theory on the given manifold. %Although Short-time existence and well-posedness for (\ref{ymf}) on complete manifolds is still lacking in the literature, and we leave this for future work. %leave this question for elsewhere.
%Such a theory for (\ref{ymf}) on complete manifolds is still lacking in the literature; 
A detailed treatment of short-time existence and well-posedness for \eqref{ymf} on complete manifolds is still lacking in the literature, and we plan to address this question separately. %; we
%we leave this question for elsewhere. 
%leave this for a separate paper.

%\subsection{Related work on Yang-Mills flow on noncompact manifolds}

\subsection*{Acknowledgements} The authors thank Claude LeBrun for pointing them to the work of Viaclovsky \cite{viaclovskyorbifoldyamabe}, which is crucial for the present paper.

\vspace{10mm}

\section{Sobolev constant and elliptic gap theorem}

\begin{comment}
    Let $(M,g)$ be a gravitational instanton, i.e. a complete hyperK\"ahler 4-manifold, with Euclidean volume growth. The Ricci tensor and the self-dual Weyl curvature both vanish. According to Bando-Kasue-Nakajima \cite{bandokasuenakajima}, $(M,g)$ is in fact ALE of order $4,$ i.e. there exists a compact subset $\Sigma \subset M,$ a finite subgroup $\Gamma \subset \SU(2),$ and diffeomorphism $\varphi : (\R^4 \setminus \{0\})/\Gamma \to M \setminus \Sigma$ such that for each integer $k \geq 0,$ we have as $r := \operatorname{dist}_{g_{\operatorname{Euc}}} (x, 0) \to \infty,$
    \begin{align*}
        \left|\nabla^{\operatorname{Euc}, k}(\varphi^\ast(g) - g_{\operatorname{Euc}})\right|_{g_{\operatorname{Euc}}} = O(r^{-4-k}).
    \end{align*}  
    %and a  asymptotic to $\R^4 /\Gamma$ for $\Gamma \subset SU(2),$ and oriented such that the self-dual Weyl tensor vanishes. See \cite[Def. 7.2.1]{joycebook} for a precise description of the requisite asymptotic behavior of $g.$ \par
Kronheimer \cite{kronheimerconstruction, kronheimertorelli} gave a complete construction of all such $M,$ which among other things implies that $\Sigma$ can be taken to be a set of measure zero (we do not need this fact).
\end{comment}
Let $(M,g)$ be a complete, oriented, Ricci-flat, anti-self-dual, ALE 4-manifold. Note that we do not assume $M$ is globally hyperK{\"a}hler, or equivalently that $M$ is simply-connected. According to Bando-Kasue-Nakajima \cite{bandokasuenakajima}, $(M,g)$ is in fact ALE of order $4,$ i.e. there exists a compact subset $\Sigma \subset M,$ a finite subgroup $\Gamma \subset \SO(4),$ and diffeomorphism $\varphi : (\R^4 \setminus \{0\})/\Gamma \to M \setminus \Sigma$ such that for each integer $k \geq 0,$ we have as $r := \operatorname{dist}_{g_{\operatorname{Euc}}} (x, 0) \to \infty,$
    \begin{align*}
        \left|\nabla^{\operatorname{Euc}, k}(\varphi^\ast(g) - g_{\operatorname{Euc}})\right|_{g_{\operatorname{Euc}}} = O\left(r^{-4-k}\right).
    \end{align*}  
    %and a  asymptotic to $\R^4 /\Gamma$ for $\Gamma \subset SU(2),$ and oriented such that the self-dual Weyl tensor vanishes. See \cite[Def. 7.2.1]{joycebook} for a precise description of the requisite asymptotic behavior of $g.$ \par
Kronheimer's classification \cite{kronheimerconstruction, kronheimertorelli} implies that $\Sigma$ can be taken to be a set of measure zero, although we do not need this fact.

%\begin{thm}[Kronheimer]\label{thm:asymptoticchart}
%    Given $M,$ there exists a measure zero set $\Sigma$ and a diffeomorphism $\varphi : (\R^4 \setminus \{0\})/\Gamma \to M \setminus \Sigma$ such that for each integer $k \geq 0,$ we have as $r := \operatorname{dist}_{g_{\operatorname{Euc}}} (x, 0) \to \infty,$
%    \begin{align*}
%        \left|\nabla^{\operatorname{Euc}, k}(\varphi^\ast(g) - g_{\operatorname{Euc}})\right|_{g_{\operatorname{Euc}}} = O(r^{-4-k}).
%    \end{align*}
%\end{thm}

\begin{defn}\label{defn:topologicalcharge}
Let $A$ be a connection on a unitary bundle $E \to M$ with $L^2$ curvature. The \emph{topological charge} of $A$ is
$$\kappa(A) = \frac{1}{8\pi^2} \int_M \Tr F_A \wedge F_A.$$
Note that
$$\kappa(A) = \frac{1}{8\pi^2} \int_M \left( |F_A^-|^2 - |F_A^+|^2 \right) \, dV_g,$$
so that $\YM(A) = 4 \pi^2 |\kappa(A)|$ if and only if $F_A = F_A^-$ or $F_A = F_A^+$ is self-dual or anti-self-dual, i.e. an instanton.
\end{defn}

If $M$ is compact then $\kappa$ depends only on the topology of the bundle $E,$ and takes integer values (equal to the second Chern class) if $c_1(E) = 0$. If $M$ is noncompact then $\kappa$ is no longer independent of $A$ or integer-valued, but connections with prescribed asymptotics have the same value mod $\Z.$ For the case of ALE 4-manifolds, the value of $\kappa$ is equivalent mod $\Z$ to the Chern-Simons invariant of the induced flat connection over $S^3 / \Gamma.$ We shall not need these facts in the present paper.

The following Sobolev inequality, in the hyperK{\"a}hler case, %follows from work of 
is originally due to Viaclovsky \cite[Theorem 1.3]{viaclovskyorbifoldyamabe}; for $M$ not globally hyperK{\"a}hler, we may cite the general result of Balogh-Kristaly \cite[Theorem 1.2]{sharpsobolevnonnegativericci}), which holds for complete manifolds with nonnegative Ricci curvature and maximal volume growth.
\begin{thm}\label{thm:sobolevineq}
    For $f \in \dot{H}^1(M),$ we have
    \begin{equation}\label{sobolevineq:ineq}
        \|f\|_{L^4}^2 \leq \frac{(6|\Gamma|)^{\frac{1}{2}}}{8\pi}\|\nabla f \|_{L^2}^2.
    \end{equation}
\end{thm}
\begin{comment}
\begin{thm}
    For $f \in \dot{H}^1(M),$ we have the Sobolev inequality
    \begin{equation}
        \|f\|_{L^4} \leq \frac{(6|\Gamma|)^{\frac{1}{4}}}{2\sqrt{2\pi}}\|\nabla f \|_{L^2}.
    \end{equation}
    %\begin{comment}
    \begin{equation}
        \|f\|_{L^4} \leq 2\sqrt{2} \pi \|\nabla f \|_{L^2}.
    \end{equation}
    %\end{comment
\end{thm}
\end{comment}
\begin{rmk}\label{rmk:gap}
Notice that this inequality is invariant under metric dilations $g \to \lambda g,$ for constant $\lambda > 0.$ Furthermore, we crucially have that the constant in (\ref{sobolevineq:ineq}) is the $\dot{H}^1$ Sobolev constant of $\R^4/\Gamma.$
\end{rmk}

From this inequality we can deduce the following gap theorem using a version of a recent argument by Vieira \cite{vieiragap}.

\begin{thm}\label{thm:gap}
    For a finite-energy Yang-Mills connection $A$ on an $\SU(r)$-bundle $E \to M,$ we have the implication
    $$\|F_A^+ \|^2_{L^2} < \frac{8 \pi^2}{|\Gamma|} \quad \Rightarrow \quad F^+_A \equiv 0.$$
\end{thm}
\begin{proof}
    Since $(M, g)$ has vanishing scalar and self-dual Weyl curvatures, the Weitzenb{\"o}ck formula for $\Omega_+^2(\mathfrak{g}_E)$ implies
    \begin{align*}
        \frac{1}{2}\nabla^\ast \nabla |F_A^+|^2 + |\nabla_A F_A^+|^2 = \langle [F_A^+, F_A^+], F_A^+\rangle.
    \end{align*} Since $D_A^\ast F_A^+ = D_A F_A^+ = 0,$ the improved Kato inequality \cite[Lemma 3.1]{radedecay} implies
    \begin{align*}
        6|F^+|\left|\nabla |F_A^+|^{\frac{1}{2}}\right|^2 = \frac{3}{2}|\nabla |F_A^+||^2 \leq |\nabla_A F_A^+|^2.
    \end{align*} For structure group $SU(r),$ we have (see \cite[Remark 3.2 (4)]{gappaper})
\begin{align*}
    |\langle [F_A^+, F_A^+], F_A^+\rangle| \leq \frac{4}{\sqrt{3}}|F_A^+|^3.
\end{align*}
Hence
\begin{align*}
    2|F_A^+|^{\frac{3}{2}}\nabla^\ast \nabla |F_A^+|^{\frac{1}{2}} &= \frac{1}{2}\nabla^\ast \nabla |F_A^+|^2 + 6|F_A^+|\left|\nabla |F_A^+|^{\frac{1}{2}}\right|^2 \\
    &\leq 
    \frac{1}{2}\nabla^\ast \nabla |F_A^+|^2 + |\nabla_A F_A^+|^2 \\
    &\leq \frac{4}{\sqrt{3}} |F_A^+|^3.
\end{align*} We divide both sides by $2|F_A^+|,$ integrate over $M,$ and integrate by parts on the left-hand side to obtain
\begin{align*}
    \left\|\nabla |F_A^+|^{\frac{1}{2}}\right\|_2^2 \leq \frac{2}{\sqrt{3}} \left\||F_A^+|^{\frac{1}{2}}\right\|_4^4. 
\end{align*} By H{\"o}lder's inequality,
\begin{align*}
    \left\||F_A^+|^{\frac{1}{2}}\right\|_4^4 \leq \|F_A^+\|_2\|F_A^+\|_4^2, 
\end{align*} so if 
\begin{align*}
    \|F_A^+\|_2^2 < \frac{8\pi^2}{|\Gamma|},
\end{align*} then
\begin{align*}
     \left\|\nabla|F_A^+|^{\frac{1}{2}}\right\|_2^2 < \frac{8\pi}{(6|\Gamma|)^{\frac{1}{2}}} \|F_A^+\|_4^2.
\end{align*} This violates the Sobolev inequality (\ref{sobolevineq:ineq}) unless $F_A^+ \equiv 0,$ so the lemma holds.
\end{proof}

\vspace{10mm}

\section{Basic energy identities}

In this section, we record the basic estimates used throughout the paper. Further details can be found in \cite{gappaper, instantons}. In the sequel, $A(t)$ is a solution of (\ref{ymf}) on $M.$ We abbreviate
\begin{align*}
    F = F(t) = F_{A(t)},
\end{align*} and similarly for $F^+, D^\ast F,$ and so on. We use $C$ to denote a positive, universal constant that may increase in each appearance.

\subsection{Energy identities}
%\subsection{Global energy identity}

We show that if $\mathcal{YM}(A(t))$ is locally bounded in time, then in fact $\mathcal{YM}(A(t))$ is nonincreasing. For these lemmas, $M$ only needs to be a complete (oriented) Riemannian 4-manifold. First we recall the following local energy estimate.

\begin{lemma}\label{lemma:energyineq} %[Split local energy inequality]
    Let $\varphi \in C_c^\infty(M).$ For any times $a < b,$ we have
    \begin{align}\label{energyineq:ineq}
        \left|\int \varphi^2(|F^+(b)|^2  - |F^+(a)|^2) + \int_a^b \hspace{-0.3em} \int \hspace{-0.3em} \varphi^2 |D^\ast F|^2\right| \leq 4\left(\int_a^b \hspace{-0.5em}\int_{\text{supp}(\nabla \varphi)} \hspace{-0.3em}\varphi^2 |D^\ast F|^2 \right)^{\frac{1}{2}}\left(\int_a^b\hspace{-0.3em}\int \hspace{-0.3em} |\nabla \varphi|^2|F^+|^2\right)^{\frac{1}{2}}.
    \end{align} The inequality also holds with $F^-$ replacing $F^+.$
\end{lemma}
\begin{proof}
    Following the proof of \cite[Lemma 3.4]{gappaper}, we have
    \begin{align*}
        \frac{d}{dt} \int \varphi^2 |F^+|^2 + \int \varphi^2 |D^\ast F|^2 = 4 \int \langle \varphi D^\ast F, \nabla \varphi \intprod F^+\rangle.
    \end{align*} We integrate in time from $t = a$ to $t = b.$ By H{\"o}lder's inequality,
    \begin{align*}
        \left|\int_a^b \int \langle \varphi D^\ast F, \nabla \varphi \intprod F^+\rangle \right| &= \left|\int_a^b \int_{\text{supp}(\nabla \varphi)} \langle \varphi D^\ast F, \nabla \varphi \intprod F^+\rangle \right| \\
        &\leq \int_a^b \int_{\text{supp}(\nabla \varphi)} \varphi |D^\ast F| |\nabla \varphi||F^+| \\
        &\leq \left(\int_a^b \int_{\text{supp}(\nabla \varphi)} \varphi^2 |D^\ast F|^2\right)^{\frac{1}{2}}\left(\int_a^b \int_{\text{supp}(\nabla \varphi)}|\nabla \varphi|^2 |F^+|^2\right)^{\frac{1}{2}}.
    \end{align*} Thus (\ref{energyineq:ineq}) follows.
\end{proof}
\begin{lemma}\label{lemma:energyidentity}
    Let $T \in (0, \infty).$ Suppose $A(t)$ is a solution of (\ref{ymf}) on $[0, T]$ such that
    \begin{align}\label{energyidentity:assumption}
        \int_M |F(0)|^2 + \int_0^T \int_M |F|^2 < \infty.
    \end{align} Then for $t \in [0, T],$ we have the global energy identities
    \begin{align}\label{energyidentity:identity}
        \int_M |F^{\pm}(0)|^2 = \int_M |F^{\pm}(t)|^2 + \int_0^t \int_M |D^\ast F|^2
    \end{align} and the conservation of charge
    \begin{align}\label{energyidentity:chernweil}
        \kappa(A(t)) = \kappa(A(0)). %\int_M \operatorname{Tr} \left(F(t) \wedge F(t) \right) = \int \operatorname{Tr} \left(F(0) \wedge F(0)\right).
    \end{align}
\end{lemma}
\begin{proof}
    It follows from the proof of Lemma \ref{lemma:energyineq} that for any $\varphi \in C_c^\infty(M)$ and $t \in (0, T],$
    \begin{align*}
        \int \varphi^2 |F(t)|^2 + 2\int_0^t \int \varphi^2 |D^\ast F|^2 \leq \int \varphi^2 |F(0)|^2 + 4\left(\int_0^t \int \varphi^2 |D^\ast F|^2\right)^{\frac{1}{2}}\left(\int_0^t \int |\nabla \varphi|^2 |F|^2\right)^{\frac{1}{2}}.
    \end{align*} By the Peter-Paul inequality,
    \begin{align*}
        \left(\int_0^t \int \varphi^2 |D^\ast F|^2\right)^{\frac{1}{2}}\left(\int_0^t \int |\nabla \varphi|^2 |F|^2\right)^{\frac{1}{2}} \leq \frac{1}{4}\int_0^t\int \varphi^2 |D^\ast F|^2 + \int_0^t \int |\nabla \varphi|^2 |F|^2.
    \end{align*} Thus
    \begin{align*}
        \int \varphi^2 |F(t)|^2 + \int_0^t \int \varphi^2 |D^\ast F|^2 \leq \int \varphi^2 |F(0)|^2 + 4\int_0^t \int |\nabla \varphi|^2 |F|^2.
    \end{align*} Let $\psi : \R \to [0, 1]$ be a smooth cut-off function for $(-\infty, 1]$ supported in $(-\infty, 2],$ and define $\psi_n(x) := \psi(x/n).$ Fix a point $p \in M$ and denote the function $f  := \text{dist}_g(\cdot, p).$ According to \cite[Theorem 1]{smoothapproxoflipschitz}, there exists $\eta \in C^\infty(M)$ with $\|\eta - f\|_\infty \leq 1$ and $\|\nabla \eta\|_\infty \leq 2.$ We now replace $\varphi$ with $\varphi_n := \psi_n \circ \eta.$ By completeness of $M,$ the $\varphi_n$ are compactly supported, and they converge to the constant function $1$ locally uniformly with $\|\nabla \varphi_n\|_\infty = O\left(\frac{1}{n}\right)$ as $n \to \infty.$ Since
    \begin{align*}
        |\nabla \varphi_n|^2 |F|^2 \leq C|F|^2,
    \end{align*} and since we are assuming (\ref{energyidentity:assumption}), the dominated convergence theorem implies
    \begin{align*}
        \int |F(t)|^2 + \int_0^t \int |D^\ast F|^2 \leq \int |F(0)|^2. 
    \end{align*} In particular,
    \begin{align}\label{energyidentity:tensionbounded}
        \int_0^T \int |D^\ast F|^2 < \infty.
    \end{align} The energy identities (\ref{energyidentity:identity}) now follow from (\ref{energyineq:ineq}) by dominated convergence, in view of (\ref{energyidentity:tensionbounded}). \par
    We next show (\ref{energyidentity:chernweil}). We compute
    \begin{align*}
        \frac{1}{2}\frac{d}{dt}\int \varphi \text{Tr}(F \wedge F) &= - \int \varphi \text{Tr}(DD^\ast F \wedge F).
    \end{align*} By the Bianchi identity $D_A F_A = 0$ and Stokes' theorem, this last integral is equal to
    \begin{align*}
        - \int \varphi d \text{Tr}(D^\ast F \wedge F) = \int d\varphi \wedge \text{Tr}(D^\ast F \wedge F).
    \end{align*} Hence
    \begin{align*}
        \left|\int \varphi \text{Tr}(F_{A(t)} \wedge F_{A(t)}) - \int \varphi \text{Tr}(F_{A(0)} \wedge F_{A(0)})\right| &\leq C\int_0^t\int |\nabla \varphi||D^\ast F||F| \\
        &\leq C\int_0^t \int |\nabla \varphi| (|D^\ast F|^2 + |F|^2).
    \end{align*} In view of (\ref{energyidentity:assumption}) and (\ref{energyidentity:tensionbounded}), (\ref{energyidentity:chernweil}) follows from dominated convergence as before.
\end{proof}

\subsection{$\eps$-regularity}

\begin{lemma}\label{lemma:epsreg}%[Split $\eps$-regularity]
    Fix $p > 1.$ Let $x_0 \in M$ and $R > 0,$ and let $B_R(x_0)$ denote the metric ball of radius $R$ about $x_0.$ There exists $\varepsilon_0 > 0$, depending only on $|\Gamma|,$ and $C_{\ref{lemma:epsreg}} \geq 1,$ depending only on $|\Gamma|$ and $p,$ such that the following holds. %Let $B_{R}(x_0)$ denote the image of $B_R(0)$ under $\exp_{x_0}.$ 
    Suppose $A(t)$ solves (\ref{ymf}) on $B_R(x_0) \times \left[T - R^2, T \right]$ with
    \begin{align}\label{epsreg:energysmall}
        \sup_{T - R^2 \leq t \leq T} \| F^+(t) \|_{L^2(B_R(x_0))} \leq \eps_0.
    \end{align}
   % where $\eps_0$ depends on $|\Gamma|.$ 
   We then have
   \begin{equation}\label{epsreg:supbound}
        \sup_{T - \frac{R^2}{4} \leq t \leq T} \|F^+(t) \|_{L^\infty \left( B_{\frac{R}{2}} \left( x_0\right) \right)} \leq \frac{C_{\ref{lemma:epsreg}}}{R^{\frac{4}{p}}} \sup_{T - R^2 \leq t \leq T} \| F^+(t) \|_{L^p(B_R(x_0))}.
    \end{equation}
    \begin{comment}
    \begin{equation}\label{epsreg:supbound}
        \|F^+_{A} \|_{L^\infty \left( B_{R/2} \left( x_0\right) \times \left(T - \frac{R^2}{4}, T \right) \right)} \leq \frac{C_{\ref{lemma:epsreg}}}{R^{\frac{4}{p}}} \sup_{T - R^2 \leq t \leq T} \| F^+_{A(t)} \|_{L^p(B_R(x_0))}.
    \end{equation}
    \end{comment}
\end{lemma}
\begin{proof}
    Note that if $\theta(t)$ is a solution of (\ref{ymf}) on the Riemannian manifold $(N, h),$ then $\theta(\lambda^{-1} t)$ is a solution on $(N, \lambda h)$ for any $\lambda > 0.$ In particular, we may assume $A(t)$ is defined on the ball of radius $1$ about $x_0$ with respect to the metric $h := R^{-2} g$ on $M,$ for $t \in [-1, 0].$ Furthermore, since $M$ is four-dimensional, the assumption (\ref{epsreg:energysmall}) is preserved, and we have $F^{+_g} = F^{+_h} = F^+.$
    Denote
    \begin{align*}
        u := |F^+|_h.
    \end{align*} Since $(M, \lambda g)$ has vanishing scalar and self-dual Weyl curvatures for any $\lambda > 0$, $u$ satisfies
    \begin{align}\label{epsreg:diffineq}
        \partial_t u - \Delta u \leq Cu^2.
    \end{align} Note that there is no linear term on the right-hand side, c.f. \cite[(2.1)]{instantons}. The standard Schoen-Uhlenbeck $\eps$-regularity argument \cite[Prop. 2.3]{instantons} relies only on a scale-invariant Sobolev inequality (\ref{sobolevineq:ineq}), and therefore goes through in this situation. We provide the following brief alternative argument.
    
    Define a function $\varphi$ by 
    \begin{align*}
        \varphi(x) := 
        \begin{cases}
            1 & \text{dist}_h(x, x_0) \leq \frac{3}{4} \\
            7 - 8\text{dist}_h(x, x_0) & \frac{3}{4} \leq \text{dist}_h(x, x_0) \leq \frac{7}{8}\\
            0 & \frac{7}{8} \leq \text{dist}_h(x, x_0).
        \end{cases}
    \end{align*} We multiply (\ref{epsreg:diffineq}) by $2\varphi^2 u.$ Since $M$ is complete and $\varphi u$ extends by zero to all of $M,$ we may integrate by parts to obtain as in the proof of \cite[Lemma 19.1]{ligeomanalysis}
    \begin{align*}
        \frac{d}{dt} \int_M \varphi^2 u^2 + 2\int_M |\nabla(\varphi u)|^2 &\leq C\int_M \varphi^2 u^3 + 2\int_M |\nabla \varphi|^2 u^2.
    \end{align*} We integrate in time and apply the Sobolev inequality (\ref{sobolevineq:ineq}) and H{\"o}lder's inequality to obtain
    \begin{align*}
        \int_M \varphi^2 u^2(t) + \frac{16\pi}{(6|\Gamma|)^{\frac{1}{2}}}\int_{-1}^t \left(\int_M \varphi^4 u^4\right)^{\frac{1}{2}} &\leq \int_M \varphi^2 u^2(0) + C\sup_{-1 \leq s \leq t}\left(\int_{B_1(x_0)} u^2(s)\right)^{\frac{1}{2}} \int_{-1}^t\left(\int_M \varphi^4 u^4\right)^{\frac{1}{2}} \\
        &+ 2\int_{-1}^t\int_M |\nabla \varphi|^2 u^2 \\
        &\leq C\varepsilon_0\left(1 + \int_{-1}^t\left(\int_M \varphi^4 u^4\right)^{\frac{1}{2}}\right). 
    \end{align*} Thus if 
    \begin{align*}
        \varepsilon_0 \leq \frac{8\pi}{C(6|\Gamma|)^{\frac{1}{2}}},
    \end{align*} we deduce that for $t \in [-1, 0]$
    \begin{align*}
        \frac{8\pi}{(6|\Gamma|)^{\frac{1}{2}}}\int_{-1}^t \left(\int_{B_{\frac{3}{4}}(x_0)} u^4\right)^{\frac{1}{2}} &\leq C.
    \end{align*} In particular, there exists $s_0 \in [-1, -\frac{3}{4}]$ such that
    \begin{align*}
        \left(\int_{B_{\frac{3}{4}}(x_0)} u^4(s_0)\right)^{\frac{1}{2}} \leq C|\Gamma|^{\frac{1}{2}}.
    \end{align*}
   Now let $\psi$ be a cutoff for $B_{\frac{2}{3}}(x_0)$ supported in $B_{\frac{3}{4}}(x_0).$ 
    We multiply (\ref{epsreg:diffineq}) by $3\psi^2 u^2.$ Shrinking $\varepsilon_0,$ we may proceed as before, integrating in time from $s_0$ to $t,$ to obtain
    \begin{align*}
        \int_M \psi^2 u^3(t) + \frac{8\pi}{(6|\Gamma|)^{\frac{1}{2}}}\int_{s_0}^t \left(\int_M \psi^4 u^6\right)^{\frac{1}{2}} &\leq \int_M \psi^2 u^3(s_0) + 3\int_{s_0}^t\int_M |\nabla \psi|^2 u^3 \\
        &\leq \int_{B_{\frac{3}{4}}(x_0)} u^3(s_0) + C\int_{s_0}^t\int_{B_{\frac{3}{4}}(x_0)} u^3.
    \end{align*} 
    Since
    \begin{align*}
        \int_{B_{\frac{3}{4}}(x_0)} u^3 \leq \left(\int_{B_{\frac{3}{4}}(x_0)} u^2\right)^{\frac{1}{2}}\left(\int_{B_{\frac{3}{4}}(x_0)} u^4\right)^{\frac{1}{2}},
    \end{align*} we deduce that
    \begin{align*}
        \sup_{-\frac{3}{4} \leq t \leq 0} \int_{B_{\frac{2}{3}}(x_0)} u^3 \leq C\varepsilon_0|\Gamma|^{\frac{1}{2}} \leq C.
    \end{align*} Since the exponent $3$ is bigger than $2,$ we may apply standard parabolic Moser iteration (see \cite[Theorem 19.1]{ligeomanalysis}) to obtain
    \begin{align*}
        \sup_{-\frac{1}{4} < t < 0}\|u(t)\|_{L^\infty\left(B_{\frac{1}{2}}(x_0)\right)} \leq C_{|\Gamma|, p} \left(\int_{-\frac{3}{4}}^0\int_{B_{\frac{2}{3}}(x_0)} u^p\right)^{\frac{1}{p}} \leq C_{|\Gamma|, p}\sup_{-\frac{3}{4} < t < 0} \|u(t)\|_{L^p\left(B_{\frac{2}{3}}(x_0)\right)}.
    \end{align*} The estimate (\ref{epsreg:supbound}) now follows by undoing the metric rescaling.
\end{proof}

The following lemma is an adaptation of \cite[Lemma 3.4]{gappaper}.

\begin{lemma}\label{lemma:energynotsmall}
Let $E_0, R > 0,$ and let $x_0 \in M.$
    There exists a constant $c \in \left(0, \frac{1}{10}\right)$, depending only on $E_0$ and $|\Gamma|,$ as follows. Given a solution $A(t)$ satisfying 
    \begin{equation}
        \sup_{-(cR)^2 \leq t \leq 0} \|F^+(t)\|_{L^2\left(B_R(x_0)\right)} \leq E_0,
    \end{equation}
    if
    \begin{equation}\label{energynotsmall:assn}
       \left\|F^+\left( -(c R)^2 \right) \right\|_{L^2\left(B_R(x_0)\right)} < c,
       %\leq c^2,
    \end{equation}
    then
    \begin{equation}\label{energynotsmall:est}
        \sup_{ - \frac{1}{16} (cR)^2 \leq t \leq 0 } \|F^+(t)\|_{L^\infty\left(B_{\frac{1}{4}}(x_0)\right)} < \frac{1}{(c R)^2}.
        %\leq \frac{1}{c_0 r^2}.
    \end{equation}
    \begin{comment}
    \begin{equation}\label{energynotsmall:assn}
       \int_{B_{R(x_0)} \left|F^+\left( -(c R)^2 \right) \right|^2 <  c^2,
       %\leq c^2,
    \end{equation}
    then
    \begin{equation}\label{energynotsmall:est}
        \sup_{ B_{\frac{R}{4}} \times \LB - \frac{1}{16} (cR)^2, 0 \right) } |F^+| < \frac{1}{c^2 R^2}.
        %\leq \frac{1}{c_0 r^2}.
    \end{equation}
    \end{comment}
\end{lemma}
\begin{proof} 
By rescaling, we may assume without loss of generality that $R = 1.$ %parabolically rescaling by the factor $r,$ we 
The assumption (\ref{energynotsmall:assn}) then reads
\begin{equation}
    \int_{B_1(x_0)} |F^+(-c^2)|^2 < c^2.
\end{equation}
Taking $\varphi$ in (\ref{energyineq:ineq}) to be a cutoff for $B_{\frac{1}{2}} := B_{\frac{1}{2}}(x_0)$ supported in $B_1$, we have for $\tau \in (-c^2, 0)$
\begin{align*}
    \int_{B_{\frac{1}{2}}} |F^+(\tau)|^2 + \int_{-c^2}^\tau \int \hspace{-0.3em} \varphi^2 |D^\ast F|^2 \leq \int_{B_1} |F^+(-c^2)|^2 + 4\left(\int_{-c^2}^\tau \int \hspace{-0.3em}\varphi^2 |D^\ast F|^2\right)^{\frac{1}{2}}\left(\int_{-c^2}^\tau \int \hspace{-0.3em}|\nabla \varphi|^2 |F^+|^2\right)^{\frac{1}{2}}.
\end{align*} By the Peter-Paul inequality,
\begin{align*}
    \left(\int_{-c^2}^\tau \int \varphi^2 |D^\ast F|^2\right)^{\frac{1}{2}}\left(\int_{-c^2}^\tau \int |\nabla \varphi|^2 |F^+|^2\right)^{\frac{1}{2}} \leq \frac{1}{8} \int_{-c^2}^\tau \int \varphi^2 |D^\ast F|^2 + 2 \int_{-c^2}^\tau \int |\nabla \varphi|^2 |F^+|^2
\end{align*} so we have
\begin{align*}
    \int_{B_{\frac{1}{2}}} |F^+(\tau)|^2 + \frac{1}{2}\int_{-c^2}^\tau \int \varphi^2 |D^\ast F|^2 &\leq \int_{B_1} |F^+(-c^2)|^2 + 8\int_{-c^2}^\tau \int |\nabla \varphi|^2 |F^+|^2 \\
    &\leq c^2 + C\int_{-c^2}^\tau \int |F^+|^2 \\
    &\leq c^2 + c^2CE_0^2.
\end{align*} Thus if 
\begin{align*}
    c^2(1 + CE_0^2) < \varepsilon_0^2,
\end{align*} then Lemma \ref{lemma:epsreg} yields
\begin{align*}
    \sup_{B_{\frac{1}{4}} \times \left[ -\frac{1}{16} c^2, 0 \right] }|F^+| \leq \frac{\sqrt{c^2(1+CE_0^2)}}{c^2} = \frac{\sqrt{1+CE_0^2}}{c}< \frac{1}{c^2}.
\end{align*} Undoing the rescaling yields (\ref{energynotsmall:est}).
\end{proof}

\vspace{10mm}

\section{Curvature decay in $L^2 \cap L^\infty$}

%\subsection{Optimal split epsilon-regularity and energy decay}

In this section, we prove the quantitative version of Theorem \ref{thm:main}$a.$ We first establish that $F^+$ is scale-invariantly bounded along a solution of (\ref{ymf}) with self-dual energy less than the gap (\ref{main:gap}). This estimate is analogous to \cite[Theorem 3.3]{gappaper}.

\begin{thm}\label{thm:supbound}
  Given $r \in \N$ and $\gamma > 0,$ there exists $C_{\ref{thm:supbound}} \geq 1,$ depending only on $r, \gamma,$ and $(M,g),$ such that the following holds. %a bounded positive function $\kappa(t) = \kappa_{M,r,\gamma}(t)$ with $\kappa(t) \searrow 0$ as $t \to \infty,$ as follows.
    
    Let $A_0$ be a connection on an $\SU(r)$-bundle $E$ that satisfies
    \begin{equation}\label{supbound:totalenergybound}
        \mathcal{YM}(A_0) \leq \gamma^{-1}
        %< \gamma^{-1}
    \end{equation}
    and
    \begin{equation}\label{supbound:gap}
        \|F^+_{A_0} \|^2 \leq \frac{8 \pi^2}{|\Gamma|} - \gamma.
    \end{equation}
        Then any bounded-energy classical solution $A(t)$ of (\ref{ymf}) %the Yang-Mills flow
        with $A(0) = A_0$ satisfies
    \begin{align}\label{supbound:uniformbound}
        \|F^+(t)\|_\infty \leq \frac{C_{\ref{thm:supbound}}}{t}
        %<  \frac{C_{M,r,\gamma}}{t}
    \end{align}
    for $0 < t < \infty.$ \par
    \begin{comment}
    We further have
        \begin{align}\label{supbound:decay}
        \|F^+(t)\|_\infty = o\left( \frac{1}{t} \right)
    \end{align}
    as $t \to \infty.$
    \end{comment}
\end{thm}

\begin{comment}
 
\begin{lemma}
    Let $(M, g)$ be a Kronheimer ALE with group $\Gamma$ at infinity {\color{red} fix terminology}, oriented such that the self-dual Weyl curvature vanishes. Given $r \in \N$ and $\eps > 0,$ there exists $C_{M,r,\eps} >0$ as follows.
    
    Let $A_0$ be a connection on an $\SU(r)$-bundle $E$ which satisfies
    \begin{equation}
        \|F^+_{A_0} \|^2 \leq \frac{8 \pi^2}{|\Gamma|} - \eps.
    \end{equation}
        Then the solution of Yang-Mills flow starting at $A_0$ satisfies
    \begin{align*}
     t \|F^+(t)\|_\infty \leq \frac{C_{M,r,\eps}}{1 + t} + \eps
    \end{align*}
    for $0 < t < \infty.$
\end{lemma}

\end{comment}

\begin{proof}
    Suppose for contradiction the bound is false. Then there exists a sequence of solutions $A_i(t)$ of (\ref{ymf}) satisfying (\ref{supbound:totalenergybound}) and (\ref{supbound:gap}) such that 
    \begin{align*}
        \sup_{0 < t < \infty} t\|F^+(t)\|_\infty > c^{-4i-2},
    \end{align*} where $c < 1$ is the constant from Lemma \ref{lemma:energynotsmall} with $E_0 = \frac{8\pi^2}{|\Gamma|}.$ Since the $A_i(t)$ are classical solutions, we have
    \begin{align*}
        \lim_{t \searrow 0} t\|F_i^+(t)\|_\infty = 0,
    \end{align*} which implies the following. For each $i,$ and for each $j$ with $0 \leq j \leq i,$ there are times $t_{i, j} \in (0, \infty)$ such that
    \begin{align}\label{supbound:comparablesup}
        \sup_{0 \leq t \leq t_{i, j}} t\|F_i^+(t)\|_\infty = t_{i, j}\|F_i^+(t_{i, j})\|_\infty = c^{-4j-2}.
    \end{align} The assumption (\ref{supbound:gap}) and the global energy identity (\ref{energyidentity:identity}) imply
    \begin{align*}
        \int_0^{t_{i, i}} \int_M |D_i^\ast F_i|^2 \, dV_g \, dt \leq \int_M |F_i^+(0)|^2 \, dV_g \leq \frac{8\pi^2}{|\Gamma|}.
    \end{align*} Then since $t_{i, j} < t_{i, j+1}$ by construction, there exists $j_i$ with $\lceil \frac{i}{2}\rceil \leq j_i \leq i$ such that
    \begin{align}\label{supbound:tensionsmall}
        \int_{t_{i, j_i-1}}^{t_{i, j_i}} \int_M |D_i^\ast F_i|^2 \, dV_g \, dt \leq \frac{16\pi^2}{i|\Gamma|}.
    \end{align} Denote
    \begin{align*}
        \tau_i := t_{i, j_i}.
    \end{align*} By the contrapositive of Lemma \ref{lemma:energynotsmall} with $R = c^{2j_i}\sqrt{\tau_i}$, the condition
    \begin{align*}
        \tau_i\|F_i^+(\tau_i)\|_\infty \geq c^{-4j_i-2}
    \end{align*} implies that there exists $p_i \in M$ such that
    \begin{align*}
        \|F_i^+(\tau_i(1-c^{4j_i+2}))\|_{L^2(B_R(p_i))}\geq c,
    \end{align*} which further implies
    \begin{align*}
        \|F_i^+(\tau_i(1-c^{4j_i+2}))\|_{L^\infty(B_R(p_i))}\geq \frac{c}{C\tau_ic^{4j_i}},
    \end{align*} whence
    \begin{align*}
        \tau_i(1-c^{4j_i+2})\|F_i^+(\tau_i(1-c^{4j_i+2}))\|_\infty \geq (1-c^{4j_i+2})C^{-1}c^{-4j_i+1} > c^{-4(j_i-1)-2}.
    \end{align*} By definition of the $t_{i, j},$
    \begin{align*}
        t_{i, j_i-1} < \tau_i(1-c^{4j_i+2}),
    \end{align*} so that \begin{align}\label{supbound:intervalhasdefinitesize}
        \tau_i - t_{i, j_i-1} > \tau_ic^{4j_i+2}.
    \end{align}\par
    \begin{comment}
    Now, it follows from epsilon regularity that $|F_i^+(x, \tau_i)| \to 0$ uniformly in $x$ tending to spatial infinity. Thus there exist $p_i \in M$ such that
    \begin{align}\label{supbound:attainsup}
        \tau_i|F_i^+(p_i, \tau_i)| = c^{-4j_i-1}.
    \end{align}
    \end{comment}
    Denoting
    \begin{align*}
        \sigma_i := (\tau_i c^{4j_i+2})^{-1},
    \end{align*} we define metrics
    \begin{align*}
        g_i := \sigma_i g
    \end{align*} and families of connections $\theta_i(t)$ by
    \begin{align*}
        \theta_i(t) := A_i(\tau_i + \sigma_i^{-1}t).
    \end{align*} Note that $\theta_i$ solves (\ref{ymf}) on $(M, g_i) \times [-1, 0],$ and moreover (\ref{supbound:intervalhasdefinitesize}) and the scale-invariant bound (\ref{supbound:tensionsmall}) imply
    \begin{align}\label{supbound:tensiongoestozero}
        \int_{-1}^0 \int_M |D_{\theta_i}^\ast F_{\theta_i}|_{g_i}^2 \, dV_{g_i} \, dt \to 0
    \end{align} as $i \to \infty.$ Furthermore, (\ref{supbound:comparablesup}) implies
    \begin{align*}
        \tau_i(1-c^{4j_i+2}) \sup_{\tau_i(1-c^{4j_i+2}) \leq t \leq \tau_i} \|F_i^+(t)\|_{\infty, g} \leq \sup_{\tau_i(1-c^{4j_i+2}) \leq t \leq \tau_i} t\|F_i^+(t)\|_{\infty, g} \leq c^{-4j_i-2},
    \end{align*} so since
    \begin{align*}
        \sup_{-1 \leq t \leq 0} \|F_{\theta_i(t)}^+\|_{\infty, g_i} = \sigma_i^{-1} \sup_{\tau_i(1-c^{4j_i+2}) \leq t \leq \tau_i} \|F_i^+(t)\|_{\infty, g}  
    \end{align*} we deduce that
    \begin{align}\label{supbound:boundedoncylinder}
        \|F_{\theta_i}^+\|_{L^\infty(M \times [-1, 0]), g_i} \leq \sigma_i^{-1}\frac{\sigma_i}{1-c^{4j_i+2}} \leq 2.
    \end{align} The assumption (\ref{supbound:comparablesup}) further yields
    \begin{align*}
        \|F_{\theta_i(0)}^+\|_\infty \geq 1. 
    \end{align*} Then in view of the uniform bound (\ref{supbound:boundedoncylinder}), parabolic Moser iteration implies that for some small $\kappa > 0,$ there exists $q_i \in M$ and a time $s_i \in \left[-\frac{1}{2}, 0\right]$ such that
    \begin{align}\label{supbound:capturedenergy}
        \int_{B_{1}^{g_i}(q_i)} |F_{\theta_i(s_i)}^+|_{g_i}^2 \, dV_{g_i} \geq \kappa.
    \end{align}
    \begin{comment}
    Additionally, the condition (\ref{supbound:attainsup}) becomes
    \begin{align}\label{supbound:supisone}
        |F_{\theta_i(0)}^+(p_i)|_{g_i} = 1,
    \end{align} so it follows from the contrapositive of Lemma \ref{lemma:energynotsmall} that there is a small $\kappa > 0$ such that for all $i$
    \begin{align}\label{supbound:capturedenergy}
        \int_{B_{1/\kappa}^{g_i}(p_i)} |F_{\theta_i(-c/2)}^+|_{g_i}^2 \, dV_{g_i} \geq \kappa.
    \end{align} 
    \end{comment}
    Finally, the scale-invariant bounds (\ref{supbound:totalenergybound}) and (\ref{supbound:gap}) imply that for all $i$ and for all $t \in [-1, 0],$
    \begin{align}\label{supbound:sequencetotalenergybound}
        \mathcal{YM}_{g_i}(\theta_i(t)) \leq \gamma^{-1}
    \end{align} and 
    \begin{align}\label{supbound:sequencegap}
         \|F_{\theta_i(t)}^+\|_{g_i}^2 \leq \frac{8\pi^2}{|\Gamma|} - \gamma.
    \end{align}\par
     Up to passing to a subsequence (again labeled by $i$), there are three possibilities for the $\sigma_i$: \par
     \vspace{2mm}

\noindent {\bf    Case 1:} $\lim_{i \to \infty} \sigma_i = 0$. Given $\lambda > 0,$ let $\rho_\lambda$ denote the rescaling map on Euclidean space given by $\rho_\lambda(x) = \lambda^{-\frac{1}{2}} x.$ We also let $\mathbf{B}_r \subset \R^4$ denote the Euclidean ball of radius $r$ about $0,$ and we abbreviate
\begin{align*}
    U_r &:= (\R^4 \setminus \mathbf{B}_r)/\Gamma.
\end{align*} %Recall %from Kronheimer's construction Theorem \ref{thm:asymptoticchart}
%that 
Since $(M, g)$ is ALE, there is a %measure zero
compact set $\Sigma \subset M,$ a constant $r_0 > 0,$ and a diffeomorphism $\varphi : U_{r_0} \to M \setminus \Sigma$ such that, denoting
\begin{align*}
    \psi_i &:= \varphi \circ \rho_{\sigma_i} : U_{r_0\sqrt{\sigma_i}} \to M \setminus \Sigma, \\
    h_i &:= \psi_i^\ast g_i, \\
    h &:= \text{ standard flat metric on } \R^4/\Gamma,
\end{align*} $h_i$ converges to $h$ in $C^\infty(U_r, h)$ for all $r > 0.$ Up to passing to a subsequence, we have the following two subcases: \par
\emph{Subcase 1:} The points
\begin{align*}
    y_i := \psi_i^{-1}(q_i)
\end{align*} converge in the topology of $\R^4 /\Gamma$ to some point $y \in \R^4 / \Gamma,$ where we set $\psi_i^{-1}(q) = 0$ if $q \in \Sigma.$ In view of the bounds (\ref{supbound:tensiongoestozero}) and (\ref{supbound:sequencetotalenergybound}), we may apply Uhlenbeck compactness for Yang-Mills flow and Uhlenbeck's removable singularity theorem (see e.g. proof of \cite[Theorem 3.3]{gappaper} and \cite[Theorem 1.3 and \textsection 5]{waldronuhlenbeck}) as follows. We pull back the bundle $E$ and the connections $\theta_i(t)$ via $\psi_i$ to obtain solutions $\Theta_i$ of (\ref{ymf}) on bundles $E_i$ over $(U_{r_0\sqrt{\sigma_i}}, h_i) \times [-1, 0].$ After passing to a subsequence, there exists a finite collection of points $\{z_k\} \subset U_0,$ an $\operatorname{SU}(r)$-bundle $E_\infty$ over $U,$ a Yang-Mills connection $\Theta$ on $(U_0, h, E_\infty),$ an increasing exhaustion $\{V_i\}$ of $U_0 \setminus \{z_k\},$ and bundle maps $u_i : E_i|_{V_i} \to E_\infty|_{V_i},$ such that $u_i(\Theta_i)$ converges to $\Theta$ in $C_{\text{loc}}^\infty((U_0 \setminus \{z_k\}, h) \times (-1, 0]).$ \par

We claim that    \begin{align}\label{supbound:finalcapturedenergy}
        \int_{B_{2}^h(y)} |F_{\Theta}^+|_h^2 \, dV_h \geq \frac{\kappa}{4}.
    \end{align} To see this, note that $\lim_{i \to \infty} \text{vol}_{g_i}(\Sigma) = 0$ since $\lim_{i \to \infty} \sigma_i = 0.$ Thus by the $L^\infty$ bound (\ref{supbound:boundedoncylinder}) and the captured energy (\ref{supbound:capturedenergy}), we have for $i$ large enough that
    \begin{align*}
        \int_{B_1^{g_i}(q_i) \setminus \Sigma} |F_{\theta_i(s_i)}^+|_{g_i}^2 \, dV_{g_i} &= \int_{B_1^{g_i}(q_i)} |F_{\theta_i(s_i)}^+|_{g_i}^2 \, dV_{g_i} - \int_{B_1^{g_i}(q_i) \cap \Sigma} |F_{\theta_i(s_i)}^+|_{g_i}^2 \, dV_{g_i} \\
        &\geq \kappa - 4\text{vol}_{g_i}(\Sigma) \\
        &\geq \frac{\kappa}{2}.
    \end{align*} Hence
    \begin{align*}
        \int_{\psi_i^{-1}\left(B_1^{g_i}(q_i)\setminus \Sigma\right)} \left|F_{\Theta_i(s_i)}^+\right|_{h_i}^2 \, dV_{h_i} \geq \frac{\kappa}{2}.
    \end{align*} Next note that for $i$ large enough, we have for some $x_0 \in M$ that
    \begin{align*}
        \psi_i\left(B_\kappa^h(0) \cap U_{r_0\sqrt{\sigma_i}}\right) \subset B_{2\kappa\sigma_i^{-\frac{1}{2}}}^g(x_0).
    \end{align*} Thus for $i$ large enough,
    \begin{align*}
        \text{vol}_{h_i}\left(B_\kappa^h(0) \cap U_{r_0\sqrt{\sigma_i}}\right) \leq C\kappa^4.
    \end{align*} Then since $B_2^h(y) \supset \psi_i^{-1}\left(B_1^{g_i}(q_i)\setminus \Sigma\right)$ for $i$ large enough, and since $\|F_{\Theta_i(s_i)}^+\|_{\infty, h_i} \leq 2,$ we have
    \begin{align*}
       \int_{B_2^h(y) \setminus B_\kappa^h(0)} |F_\Theta^+|_h^2 \, dV_h &\geq \int_{\psi_i^{-1}\left(B_1^{g_i}(q_i)\setminus \Sigma\right) \setminus B_\kappa^h(0)} |F_\Theta^+|_h^2 \, dV_h \\
       &= \lim_{i \to \infty} \int_{\psi_i^{-1}\left(B_1^{g_i}(q_i)\setminus \Sigma\right) \setminus B_\kappa^h(0)} \left|F_{\Theta_i(s_i)}^+\right|_{h_i}^2 \, dV_{h_i} \\
       &= \lim_{i \to \infty} \int_{\psi_i^{-1}\left(B_1^{g_i}(q_i)\setminus \Sigma\right)} \left|F_{\Theta_i(s_i)}^+\right|_{h_i}^2 \, dV_{h_i} - \int_{\psi_i^{-1}\left(B_1^{g_i}(q_i)\setminus \Sigma\right) \cap B_\kappa^h(0)} \left|F_{\Theta_i(s_i)}^+\right|_{h_i}^2 \, dV_{h_i} \\
       &\geq \frac{\kappa}{2} - C\kappa^4 \\
       &\geq \frac{\kappa}{4}
    \end{align*} for $\kappa$ small enough. This establishes (\ref{supbound:finalcapturedenergy}). \par 
    \begin{comment}
        Since $\lim_{i \to \infty} \text{vol}_{g_i}(\Sigma) = 0,$ we have in view of the $L^\infty$ bound (\ref{supbound:boundedoncylinder}) that
    \begin{align*}
        \int_{B_{1}^{h}(y)} |F_{\Theta}^+|_{h}^2 \, dV_{h} = \lim_{i \to \infty} \int_{B_{1}^{g_i}(q_i)} |F_{\Theta_i(s_i)}^+|_{g_i}^2 \, dV_{g_i}.
    \end{align*} Then by (\ref{supbound:capturedenergy}),
    \begin{align}\label{supbound:finalcapturedenergy}
        \int_{B_{1}(y)} |F_{\Theta}^+|^2 \geq \kappa.
    \end{align}
    \end{comment}
    On the other hand, (\ref{supbound:sequencegap}) implies $\|F_{\Theta}^+\|_2^2 < \frac{8\pi^2}{|\Gamma|}.$ As noted in Remark \ref{rmk:gap}, $(M, g)$ and $\R^4/\Gamma$ have the same $\dot{H}^1$ Sobolev constant. Therefore $F_{\Theta}^+ \equiv 0$ by Vieira's gap Theorem \ref{thm:gap}, which contradicts (\ref{supbound:finalcapturedenergy}), as desired. \par
\emph{Subcase 2:} The points $y_i$ escape to infinity. Since the $h_i$ converge to $h$ uniformly smoothly on the end of $U,$ we have the following. There exist $r_i \nearrow \infty$ such that the exponential map
\begin{align*}
    \exp_{y_i} : \mathbf{B}_{r_i} \to B_{r_i}^{h_i}(y_i)
\end{align*} is a diffeomorphism, so we may pull back $E$ and $\theta_i$ via $\psi_i \circ \exp_{y_i}$ to obtain $E_i$ and $\Theta_i$ as before over $(\mathbf{B}_{r_i}, (\psi_i \circ \exp_{y_i})^\ast(g_i)) \times [-1, 0].$ As before, we obtain an Uhlenbeck limit $\Theta,$ now on $\R^4,$ satisfying (\ref{supbound:finalcapturedenergy}). Since $\frac{8\pi^2}{|\Gamma|} \leq 8\pi^2,$ $\Theta$ is anti-self-dual by the elliptic gap theorem, contradicting (\ref{supbound:finalcapturedenergy}). \par Thus Case 1 ends in a contradiction. \par

    \vspace{2mm}
    
\noindent {\bf Case 2:} $\lim_{i \to \infty} \sigma_i = \sigma \in (0, \infty)$. We again have two subcases. \par
\emph{Subcase 1:} The points $q_i$ converge to some point in $M.$ We obtain an Uhlenbeck limit on $(M, \sigma g)$ itself which is anti-self-dual by Theorem \ref{thm:gap}, contradicting (\ref{supbound:finalcapturedenergy}). \par
\emph{Subcase 2:} Otherwise, the $q_i$ escape to infinity. Since the metric $\sigma g$ approaches Euclidean on the end, we obtain an Uhlenbeck limit on $\R^4$ and obtain a contradiction as in Subcase 2 of Case 1. \par
Thus Case 2 ends in a contradiction.
\vspace{2mm}

{\bf Case 3:} $\lim_{i \to \infty} \sigma_i = \infty.$ By the bounded geometry of $(M, g),$ this choice of scaling yields an Uhlenbeck limit on $\R^4$ as in the proof of \cite[Theorem 3.3]{gappaper}, and we reach a contradiction as before. \par

\vspace{2mm} 

    Since Cases 1-3 are exhaustive and we ultimately reach a contradiction, the desired estimate (\ref{supbound:uniformbound}) must hold.
\end{proof}
Next we show that the self-dual energy converges to zero as $t \to \infty.$
\begin{lemma}\label{lemma:energygoestozero}
    Let $\varepsilon > 0.$ There exists $\sigma > 0,$ depending only on $\varepsilon$ and $C_{\ref{thm:supbound}},$ and $C_{\ref{lemma:energygoestozero}} > 0,$ depending only on $|\Gamma|$ and $C_{\ref{thm:supbound}},$ %r, \gamma,$ and $(M,g),$ 
    such that the following holds. Suppose $A(t)$ is a solution of (\ref{ymf}) satisfying the hypotheses of Theorem \ref{thm:supbound}. If for some $x \in M$ and $R > 0,$
    \begin{align*}
        \int_{M \setminus B_R(x)} |F^+(0)|^2 \leq \varepsilon,
    \end{align*} then for all $t \geq \sigma R^2$
    \begin{align*}
        \int_M |F^+(t)|^2 \leq 2\varepsilon,
    \end{align*}
    and for all $t \geq 2\sigma R^2$
        \begin{align*}
        \|F^+(t)\|_{L^\infty(M)} \leq \frac{C_{\ref{lemma:energygoestozero}} \sqrt{\varepsilon}}{t}.
    \end{align*}
\end{lemma}
\begin{proof}
    First we estimate the self-dual energy in the complement of a large ball. Set 
    \begin{align*}
        \delta := \frac{\varepsilon}{4CC_{\ref{thm:supbound}}}.
    \end{align*} Since
    \begin{align*}
        \int_0^\infty \int_M |D^\ast F|^2 \leq \frac{8\pi^2}{|\Gamma|},
    \end{align*} there exists an integer
    \begin{align*}
        j \in [0, 8\pi^2(\delta^2|\Gamma|)^{-1}+1]
    \end{align*} such that, denoting
    \begin{align*}
        \rho &:= 2^j R \\
        U &:= B_{2\rho}(x) \setminus B_{\rho}(x),
    \end{align*} we have
    \begin{align*}
        \int_0^\infty \int_U |D^\ast F|^2 \leq \delta^2.
    \end{align*} Define a function $\varphi$ by
    \begin{align*}
        \varphi(y) :=
        \begin{cases}
            0 & 0 \leq \text{dist}(y, x) \leq \rho \\
            \frac{\text{dist}(y, x) - \rho}{\rho} & \rho \leq \text{dist}(y, x) \leq 2\rho \\
            1 & 2\rho \leq \text{dist}(y, x).
        \end{cases}
    \end{align*} Our assumptions on $A(t)$ allow us to take (\ref{energyineq:ineq}) with this choice of $\varphi,$ yielding for all $T \geq 0$
    \begin{align*}
        \int_{M \setminus B_{2\rho}(x)} |F^+(T)|^2 \leq \int_{M \setminus B_{\rho}(x)} |F^+(0)|^2 + 4\left(\int_0^T \int_U |D^\ast F|^2\right)^{\frac{1}{2}}\left(\int_0^T \int_U |\nabla \varphi|^2 |F^+|^2\right)^{\frac{1}{2}}.
    \end{align*} By (\ref{supbound:uniformbound}),
    \begin{align*}
        \int_0^T \int_U |\nabla \varphi|^2 |F^+|^2 &= \int_0^{\rho^2} \int_U |\nabla \varphi|^2 |F^+|^2 + \int_{\rho^2}^T \int_U |\nabla \varphi|^2 |F^+|^2 \\
        &\leq \frac{1}{\rho^2} \left(\int_0^{\rho^2} \int_U |F^+|^2 + \text{Vol}(B_{2\rho}(x))\int_{\rho^2}^T  \frac{C_{\ref{thm:supbound}}^2}{t^2} \, dt \right) \\
        &\leq 
        \frac{1}{\rho^2}\left(\rho^2 \frac{8\pi^2}{|\Gamma|} + C (2\rho)^4 C_{\ref{thm:supbound}}^2 \frac{1}{\rho^2}\right)  \\
        &\leq CC_{\ref{thm:supbound}}^2.
    \end{align*} Hence
    \begin{align*}
        \int_{M \setminus B_{2\rho}(x)} |F^+(T)|^2 \leq \varepsilon + CC_{\ref{thm:supbound}}\delta \leq \frac{3}{2}\varepsilon.
    \end{align*} On the other hand, if
    \begin{align*}
        T \geq \frac{8CC_{\ref{thm:supbound}}}{\sqrt{\varepsilon}}\rho^2
    \end{align*} then (\ref{supbound:uniformbound}) yields
    \begin{align*}
        \int_{B_{2\rho}(x)} |F^+(T)|^2 \leq C(2\rho)^4\frac{C_{\ref{thm:supbound}}^2}{T^{2}} \leq \frac{\varepsilon}{2}.
    \end{align*} Thus by definition of $\rho,$ the lemma holds with
    \begin{align*}
        \sigma = \frac{8CC_{\ref{thm:supbound}}}{\sqrt{\varepsilon}}2^{\frac{16\pi^2}{\delta^2|\Gamma|}+2}.
    \end{align*} \par
    The $L^\infty$ bound of this lemma now follows from parabolic Moser iteration, since by (\ref{supbound:uniformbound}) we have that $u := |F^+|$ satisfies
    \begin{align*}
        \partial_t u - \Delta u \leq Cu^2 \leq \frac{CC_{\ref{thm:supbound}}}{t} u 
    \end{align*} for $t > 0.$
\end{proof}

\begin{proof}[Proof of Theorem \ref{thm:main}$a$]
Since $F^+(0) \in L^2,$ we have
\begin{align*}
    \lim_{R \to \infty} \int_{M \setminus B_R(x)} |F^+(0)|^2 = 0
\end{align*}
for any $x \in M.$ 
Then Lemma \ref{lemma:energygoestozero} implies $\lim_{t \to \infty} \left(\|F^+(t)\|_2 + t\|F^+(t)\|_\infty\right) = 0.$
\end{proof}

\vspace{10mm}

\section{Improved decay for $F^+ \in L^p, p \in (1, 2)$}

In this section, we prove the quantitative version of Theorem \ref{thm:main}$b.$ %, we show that if we further have $F^+(0) \in L^p$ for some $p \in (1, 2),$ then the decay of $\|F^+(t)\|_\infty$ improves to $O\left(t^{-\frac{2}{p}}\right)$ for large $t.$
The argument has two parts: we first show that once $\|F^+(t)\|_2$ is small, $\|F^+(t)\|_p$ remains bounded and thus $\|F^+(t)\|_\infty$ has improved decay. We then show that there is a uniform time after which $\|F^+(t)\|_2$ is small, so that $\|F^+(t)\|_p$ thereafter remains bounded.
\begin{comment}
In this section, we show that if we further have $F^+(0) \in L^p$ for some $p \in (1, 2),$ then the decay of $\|F^+(t)\|_\infty$ improves to $O(t^{-\frac{2}{p}})$ for large $t.$ The argument has two parts: we show that if $\|F^+(t)\|_p$ is bounded, then $\|F^+(t)\|_\infty$ has improved decay, and we show that there is a uniform time after which $\|F^+(t)\|_p$ remains bounded.
\end{comment}
\begin{lemma}\label{lemma:Lpcontrolafter}
Let $A(t)$ be a solution of (\ref{ymf}), and let $p > 1$ and $t_0 \geq 0.$ There exists $\varepsilon_p > 0,$ depending only on $p$ and $|\Gamma|,$ such that if 
\begin{align*}
    \|F^+(t_0)\|_2 \leq \eps_p,
\end{align*} then
$$\|F^+(t)\|_{p} \leq \|F^+(t_0)\|_p$$
for $t \geq t_0.$
%monotonically decreasing.
\end{lemma}
\begin{proof}
    Denote
    \begin{align*}
        u := |F^+|.
    \end{align*} Recall that $u$ satisfies
    \begin{align*}
        \partial_t u - \Delta u \leq Cu^2
    \end{align*}
    We multiply the inequality by $pu^{p-1},$ integrate over $M$, and then integrate by parts to obtain
    \begin{align}\label{Lpcontrolafter:firststepofMoser}
        \int \partial_t(u^p) + \frac{4(p-1)}{p}\int \left|\nabla u^{\frac{p}{2}}\right|^2 \leq pC\int u^{p+1}.
    \end{align} By the Sobolev inequality (\ref{sobolevineq:ineq}),
    \begin{align*}
        \frac{8\pi}{(6|\Gamma|)^{\frac{1}{2}}}\left(\int u^{2p}\right)^{\frac{1}{2}}  \leq \int \left|\nabla\left(u^{\frac{p}{2}}\right)\right|^2.
    \end{align*} By H{\"o}lder's inequality, we have for $t \geq t_0$ that
    \begin{align*}
        \int u^{p+1} \leq \|u\|_2\|u^p\|_2 \leq \varepsilon_p\left(\int u^{2p}\right)^{\frac{1}{2}}.
    \end{align*} Thus if
    \begin{align*}
        \varepsilon_p \leq \frac{16\pi(p-1)}{p^2C(6|\Gamma|)^{\frac{1}{2}}},
    \end{align*} then for $t \geq t_0$
    \begin{align*}
        \frac{d}{dt} \int u^p + \frac{16\pi(p-1)}{p(6|\Gamma|)^{\frac{1}{2}}}\left(\int u^{2p}\right)^{\frac{1}{2}} \leq 0. %,
    \end{align*}% so it follows that the $L^p$ norm of $F^+$ is nonincreasing. \par
    Integrating in time yields for $t \geq t_0$
    \begin{align}\label{Lpcontrolafter:monotone}
        \int u^p(t) + \frac{16\pi(p-1)}{p(6|\Gamma|)^{\frac{1}{2}}}\int_{t_0}^t \left(\int u^{2p}\right)^{\frac{1}{2}} \leq \int u^p(t_0),
    \end{align} which in particular implies the lemma.
\end{proof}
\begin{lemma}\label{lemma:improveddecay}
    Fix $K \geq 0.$ Let $A(t)$ be a solution of (\ref{ymf}) satisfying the hypotheses of Lemma \ref{lemma:Lpcontrolafter} and satisfying 
    \begin{align}\label{improveddecay:supbound}
        \|F^+(t) \|_\infty \leq \frac{K}{t}
    \end{align} for $t \geq t_0.$  There exists $C_{\ref{lemma:improveddecay}} > 0,$ depending only on $p, |\Gamma|,$ and $K,$ such that for $t \geq 2t_0,$
    \begin{align*}
        \|F^+(t)\|_\infty \leq C_{\ref{lemma:improveddecay}}\|F^+(t_0)\|_p t^{-\frac{2}{p}}.
    \end{align*}
    Moreover,
    \begin{align*}
        \lim_{t \to \infty} \left(\|F^+(t)\|_p + t^{\frac{2}{p}}\|F^+(t)\|_\infty\right) = 0.
    \end{align*}
    \begin{comment}
        Moreover,
    \begin{align*}
        \|F^+(t)\|_\infty = o\left(t^{-\frac{2}{p}}\right)
    \end{align*} as $t \to \infty.$
    \end{comment}
\end{lemma}
\begin{proof}
    By (\ref{improveddecay:supbound}), $u$ satisfies
    \begin{align*}
        \partial_t u - \Delta u \leq \frac{CK}{t}u.
    \end{align*} Thus by parabolic Moser iteration, there exists $C_0 > 0,$ depending only on $p, |\Gamma|,$ and $K,$ such that for all $(x, t) \in M \times (0, \infty)$
    \begin{align*}
        \|u(t)\|_{L^\infty\left(B_{\sqrt{t}}(x)\right)} \leq C_0t^{-\frac{3}{p}} \left(\int_{t/2}^t \int_{B_{2\sqrt{t}}(x)} u^p\right)^{\frac{1}{p}}.
    \end{align*} By H{\"o}lder's inequality,
    \begin{align*}
        \int_{t/2}^t \int_{B_{2\sqrt{t}}(x)} u^p &\leq \sqrt{\text{Vol}(B_{2\sqrt{t}}(x))}\int_{t/2}^t \left(\int_{B_{2\sqrt{t}}(x)} u^{2p}\right)^{\frac{1}{2}} \\
        &\leq Ct \int_{t/2}^t \left(\int_{B_{2\sqrt{t}}(x)} u^{2p}\right)^{\frac{1}{2}}. 
    \end{align*} By (\ref{Lpcontrolafter:monotone}),
    \begin{align*}
        \int_{t_0}^\infty\left( \int u^{2p}\right)^{\frac{1}{2}} \leq \|u(t_0)\|_p^p,
    \end{align*} and thus we further have
    \begin{align*}
    \lim_{t \to \infty} \int_{t/2}^t\left(\int_{B_{2\sqrt{t}}(x)} u^{2p}\right)^{\frac{1}{2}} = 0.
    \end{align*} The $L^\infty$ estimates of the lemma thus follow. \par 
    We next show that $\lim_{t \to \infty}\|u(t)\|_p = 0.$ For $t \geq 2t_0$, we now have the differential inequality
    \begin{align*}
        \partial_t u - \Delta u \leq CC_{\ref{lemma:improveddecay}}t^{-\frac{2}{p}} u.
    \end{align*} Hence the function
    \begin{align*}
        v := e^{\frac{pCC_{\ref{lemma:improveddecay}}}{2-p}\left(t^{1-\frac{2}{p}}-(2t_0)^{1-\frac{2}{p}}\right)} u
    \end{align*} satisfies 
    \begin{align*}
        \partial_t v - \Delta v \leq 0
    \end{align*} and 
    \begin{align*}
        v(2t_0) = u(2t_0).
    \end{align*} Let $\tilde{v}$ denote the semigroup solution of the heat equation starting at time $2t_0$ with initial data $u(2t_0).$ Then by the maximum principle, $v(t) \leq \tilde{v}(t),$ so we have
    \begin{align*}
        u(t) \leq e^{\frac{pCC_{\ref{lemma:improveddecay}}}{2-p}(2t_0)^{1-\frac{2}{p}}} \tilde{v}(t).
    \end{align*} \par
    We now use that the heat semigroup $H_t$ on $(M, g)$ satisfies for $1 \leq q \leq r \leq \infty$ and $t > 0$ 
    \begin{align}\label{improveddecay:semigroup}
        \|H_t(f)\|_r \leq C_{|\Gamma|} t^{-2\left(\frac{1}{q} - \frac{1}{r}\right)}\|f\|_q.
    \end{align} To see this, note that we have
    \begin{align*}
        \|H_t(f)\|_r \leq \|f\|_r
    \end{align*} since the heat semigroup is a contraction on $L^r,$ and we have
    \begin{align*}
        \|H_t(f)\|_\infty \leq C_{|\Gamma|}t^{-\frac{2}{q}}\sup_{t/2 \leq s \leq t} \|H_s(f)\|_q \leq C_{|\Gamma|}t^{-\frac{2}{q}} \|f\|_q
    \end{align*} by parabolic Moser iteration. The general bound (\ref{improveddecay:semigroup}) now follows by interpolation via H{\"o}lder's inequality. \par
    Fixing $x \in M$ and letting $\varphi$ be the indicator function for $B_R(x),$ we thus have
    \begin{align*}
        \|\tilde{v}(t)\|_p &\leq \|H_{t-2t_0}(\varphi\tilde{v}(2t_0))\|_p + \|H_{t-2t_0}((1-\varphi)\tilde{v}(2t_0))\|_p \\
        &\leq C_{|\Gamma|}(t-2t_0)^{-2\left(1-\frac{1}{p}\right)}\|\varphi\tilde{v}(2t_0)\|_1 + \|(1-\varphi)\tilde{v}(2t_0)\|_p \\
        &\leq C_{|\Gamma|}(t-2t_0)^{-2\left(1-\frac{1}{p}\right)}(CR^4)^{1-\frac{1}{p}}\|\tilde{v}(2t_0)\|_p + \|(1-\varphi)\tilde{v}(2t_0)\|_p.
    \end{align*} Since $\tilde{v}(2t_0) \in L^p,$ we may choose $R$ large enough to make the second term small, and then choose $t$ large enough, depending on $R, |\Gamma|,$ and $\|\tilde{v}(2t_0)\|_p,$ to make the first term small. We deduce that $\|\tilde{v}(t)\|_p \to 0,$ so the lemma follows.
    %By standard properties {\color{red}?} of the heat semigroup on an ALE space,
    %\begin{align*}
    %    \|\tilde{v}(t)\|_p \to 0,
    %\end{align*} so the lemma follows.
\end{proof}
We now work towards showing that $\|F^+(t)\|_2$ is small after a uniform time. First we show that $\|F^+(t)\|_p$ remains locally bounded.
\begin{lemma}\label{lemma:Lpcontrolbefore}
    Fix $E_0, K \geq 0$ and $p \in (1, 2).$ Suppose $A(t)$ is a solution of (\ref{ymf}) satisfying 
    \begin{align}
        \|F^+(t)\|_2 &\leq E_0 \label{Lpcontrolbefore:energybound}\\
        \|F^+(t)\|_\infty &\leq \frac{K}{t} \label{Lpcontrolbefore:supbound}
    \end{align} for $t \geq 0.$ Then    \begin{align}\label{Lpcontrolbefore:bound}
        \|F^+(t)\|_p^p \leq \|F^+(0)\|_p^p + C_{\ref{lemma:Lpcontrolbefore}}t^{2-p}
    \end{align} for $t \geq 0,$ where $C_{\ref{lemma:Lpcontrolbefore}} > 0$ depends only on $p, E_0,$ and $K.$
\end{lemma}
\begin{proof}
    By (\ref{supbound:uniformbound}) and (\ref{Lpcontrolafter:firststepofMoser}),
    \begin{align*}
        \frac{d}{dt}\int u^p \leq pC\|u\|_\infty^{p-1} \|u\|_2^2 \leq pC\left(\frac{K}{t}\right)^{p-1}E_0^2.
    \end{align*} The desired estimate follows by integrating in time.
\end{proof}
Next we show that if the space-time $L^2$ norm of $D^\ast F$ is small, then $F^+$ is scale-invariantly small in $L^\infty,$ not just bounded.
\begin{lemma}\label{lemma:tensioncontrolssup}
     Let $\varepsilon, t > 0.$ Suppose $A(t)$ satisfies the hypotheses of Lemma \ref{lemma:Lpcontrolbefore}. There exist $\Omega \geq 1,$ depending only on $\varepsilon$ and $K,$ and $\delta > 0,$ depending only on $\varepsilon, E_0,$ and $K,$ such that the following holds. If
     \begin{align*}
         \int_{t/2}^{\Omega t} \int |D^\ast F|^2 \leq \delta^2,
     \end{align*} then
     \begin{align*}
         \|F^+(t)\|_\infty \leq \frac{\varepsilon}{t}.
     \end{align*}
\end{lemma}
\begin{proof}
    We may assume $\varepsilon <\varepsilon_0.$ Let $q \in M.$ Then for $\Omega$ large enough depending on $\varepsilon$ and $K,$ (\ref{Lpcontrolbefore:supbound}) implies
    \begin{align*}
        \int_{B_{2\sqrt{t}}(q)} |F^+(\Omega t)|^2 \leq \frac{K^2}{(\Omega t)^2}C\left(2\sqrt{t}\right)^4 \leq \frac{\varepsilon^2}{4C^2}.
    \end{align*} Let $\varphi$ be the cutoff function
    \begin{align*}
        \varphi(x) :=
        \begin{cases}
            1 & \text{dist}(x, q) \leq  \sqrt{t} \\
            \frac{2\sqrt{t}-\text{dist}(x, q)}{\sqrt{t}} & \sqrt{t} \leq \text{dist}(x, q) \leq 2\sqrt{t} \\
            0 & 2\sqrt{t} \leq \text{dist}(x, q).
        \end{cases}
    \end{align*}Then by (\ref{energyineq:ineq}) with this choice of $\varphi,$ we have for $\tau \in [t/2, \Omega t]$
    \begin{align*}
         \int_{B_{\sqrt{t}}(q)} |F^+(\tau)|^2 &\leq \int_{B_{2\sqrt{t}}(q)} |F^+(\Omega t)|^2 + \int_{t/2}^{\tau}\hspace{-0.2em}\int |D^\ast F|^2 + 4\left(\int_{t/2}^{\tau} \hspace{-0.2em} \int |D^\ast F|^2\right)^{\frac{1}{2}}\left(\int_{t/2}^{\tau} \hspace{-0.2em} \int |\nabla \varphi|^2 |F^+|^2\right)^{\frac{1}{2}} \\
         &\leq \frac{\varepsilon^2}{4C^2} + \delta^2 + C\delta E_0\sqrt{\Omega} \\
         &\leq (\varepsilon/C)^2,
    \end{align*} provided $\delta > 0$ is small enough depending on $\varepsilon, E_0,$ and $K.$ Since $q \in M$ was arbitrary, epsilon-regularity yields
    \begin{align*}
        \|F^+(t)\|_\infty \leq \frac{\varepsilon}{t},
    \end{align*} as desired.
\end{proof}

We now establish a uniform time after which $\|F^+(t)\|_2$ is less than $\varepsilon_p.$

\begin{lemma}\label{lemma:uniformtime}
    Suppose $A(t)$ is a solution of (\ref{ymf}) satisfying the hypotheses of Lemma \ref{lemma:Lpcontrolbefore}. There exists $t_0 > 0,$ depending only on $p, E_0, K, |\Gamma|,$ and $\|F^+(0)\|_{p},$ such that $\| F^+(t_0)\|_2 \leq \eps_p.$
\end{lemma}

\begin{proof}
    Take $\Omega, \delta$ from Lemma \ref{lemma:tensioncontrolssup} such that the lemma holds with
    \begin{align*}
        \varepsilon = \min\left\{\left(\frac{\varepsilon_p}{\sqrt{\|u(0)\|_p^p + C_{\ref{lemma:Lpcontrolbefore}}}}\right)^{\frac{2}{2-p}}, 1\right\}
    \end{align*} By (\ref{energyidentity:identity}) and (\ref{Lpcontrolbefore:energybound}),
    \begin{align*}
        \int_0^\infty\int_M |D^\ast F|^2 \leq E_0^2,
    \end{align*} so there exists an integer
    \begin{align*}
        i \in [0, E_0^2\delta^{-2} + 1]
    \end{align*} such that
    \begin{align*}
        \int_{(2\Omega)^i/2}^{\Omega(2\Omega)^i} \int_M |D^\ast F|^2 \leq \delta^2.
    \end{align*} Denoting
    \begin{align*}
        t_0 := (2\Omega)^i \geq 1,
    \end{align*} we have by Lemma \ref{lemma:Lpcontrolbefore} that
    \begin{align*}
        \|u(t_0)\|_2 &\leq \|u(t_0)\|_p^{\frac{p}{2}}\|u(t_0)\|_\infty^{1-\frac{p}{2}} \\
        &\leq \left(\|u(0)\|_p^p + C_{\ref{lemma:Lpcontrolbefore}}t_0^{2-p}\right)^{\frac{1}{2}}\left(\frac{\varepsilon}{t_0}\right)^{1-\frac{p}{2}} \\
        &\leq \varepsilon_p.
    \end{align*} The desired claim now follows from Lemma \ref{lemma:Lpcontrolafter}.
\end{proof}

Finally, we have the following uniform estimate on the $L^p$ norm.

\begin{thm} \label{thm:Lpbounded}
    Let $E_0, K \geq 0$ and $p \in (1, 2).$ Suppose $A(t)$ is a solution of (\ref{ymf}) such that for $t \geq 0$
    \begin{align}
        \|F^+(t)\|_2 &\leq E_0, \label{Lpbounded:energybound} \\
        \|F^+(t)\|_\infty &\leq \frac{K}{t}, \label{Lpbounded:supbound}
    \end{align} and $F^+(0) \in L^p.$ There exists $C_{\ref{thm:Lpbounded}} > 0,$ depending only on $p, E_0, K, |\Gamma|,$ and $\|F^+(0)\|_{p},$ such that
    \begin{align*}
       \|F^+(t)\|_p + \left( t + t^{\frac{2}{p}} \right) \|F^+\|_{\infty} \leq C_{\ref{thm:Lpbounded}}
    \end{align*}
    for all $0 \leq t < \infty.$
    Moreover, the quantity on the LHS tends to zero as $t \to \infty.$
\end{thm}
\begin{proof}

    It follows from the proofs of Lemmas \ref{lemma:tensioncontrolssup} and \ref{lemma:uniformtime} that we may take $t_0$ in Lemma \ref{lemma:Lpcontrolafter} to be less than
    \begin{align*}
        t_1 := \left(C\left(K\varepsilon^{-1}+1\right)\right)^{CE_0^2K\varepsilon^{-6}+1},
    \end{align*} where
    \begin{align*}
        \varepsilon = \min\left\{\left(\frac{\varepsilon_p}{\sqrt{\|F^+(0)\|_p^p + C_{\ref{lemma:Lpcontrolbefore}}}}\right)^{\frac{2}{2-p}}, 1\right\}.
    \end{align*} Thus by Lemmas \ref{lemma:Lpcontrolbefore} and \ref{lemma:Lpcontrolafter}, we have
    \begin{align*}
        \|F^+(t)\|_p^p \leq \|F^+(0)\|_p^p + C_{\ref{lemma:Lpcontrolbefore}} (\min\{t, t_1\})^{2-p}
    \end{align*} for $t \geq 0.$ Then for $t \geq 2t_1,$ Lemma \ref{lemma:improveddecay} yields
    \begin{align*}
        t^{\frac{2}{p}}\|F^+(t)\|_\infty \leq C_{\ref{lemma:improveddecay}},
    \end{align*} while for $t \in (0, 2t_1)$ we may use the bound (\ref{Lpbounded:supbound}). Thus the uniform estimate on $\|F^+(t)\|_p + \left(t + t^{\frac{2}{p}}\right)\|F^+(t)\|_\infty$ is established. The decay to zero follows from Lemma \ref{lemma:improveddecay}.
\end{proof}

\begin{proof}[Proof of Theorem \ref{thm:main}$b$]
    By Theorem \ref{thm:supbound}, we may take $E_0^2 = \frac{8\pi^2}{|\Gamma|}$ and $K = C_{\ref{thm:supbound}}$ in the hypotheses (\ref{Lpbounded:energybound}-\ref{Lpbounded:supbound}). The desired conclusions now follow \emph{a fortiori} from the preceding theorem.
\end{proof}

\vspace{10mm}

\section{Convergence of $A(t)$}
Finally, we record the following convergence properties of $A(t)$ under the assumptions of Theorem \ref{thm:supbound} and the preceding section. 
\begin{comment}
We remark that the assumption $F^+(0) \in L^p$ for some $p \in (1, 2)$ likely cannot be dropped in general, based on work of Gustafson, Nakanishi, and Tsai \cite{gustafsonnakanishitsai}.
\end{comment}
\begin{thm}\label{thm:convergence}
    Suppose $A(t)$ is a solution of (\ref{ymf}) satisfying
    \begin{align*}
        \|F(0)\|_2 + \|F^+(0)\|_p &< \infty, \\
        \|F^+(0)\|_2^2 &< \frac{8\pi^2}{|\Gamma|}.
    \end{align*} Then $A(t)$ converges in $C_{\text{loc}}^\infty$ to an ASD instanton $A_\infty$ as $t \to \infty.$ Moreover, $A(t)$ converges to $A_\infty$ in $L^q$ for all $q \in (\frac{2p}{2-p}, \infty],$ and $F(t)$ converges to $F_{A_\infty} = F_{A_\infty}^-$ in $L^2.$
\end{thm}
\begin{proof}
    Since
    \begin{equation}\label{F+L1}
    \|F^+(t)\|_\infty \in L^1([1, \infty)),
    \end{equation}
    it follows from \cite[Theorem 2.5]{instantons} that
    \begin{align}\label{convergence:Fbound}
        \sup_{t \in [2, \infty)} \|F(t)\|_\infty < \infty.
    \end{align} Moreover, by (\ref{energyineq:ineq}) and Lemma \ref{lemma:improveddecay}, we have for all $x \in M$ that
    \begin{align*}
        \int_{t-1}^t\int_{B_1(x)} |D^\ast F|^2 &\leq \int_{B_2(x)} |F^+(t-1)|^2 + \int_{B_2(x)} |F^+(t)|^2 + C\int_{t-1}^t\int_{B_2(x)} |F^+|^2 \\
        &= O\left(t^{-\frac{4}{p}}\right)
    \end{align*} as $t \to \infty.$ In view of the $L^\infty$ bound on the full curvature (\ref{convergence:Fbound}), we have for each integer $k \geq 0$ and for large $t$ the derivative estimates of \cite[Prop. 3.2]{instantons}
    \begin{align*}
        &\|\nabla^k D^\ast F(t)\|_\infty \lesssim \sup_{x \in M}\|D^\ast F\|_{L^2(B_1(x) \times [t-1, t])} = O\left(t^{-\frac{2}{p}}\right), \\
        &\|\nabla^k F(t)\|_\infty < \infty.
    \end{align*} Thus by \cite[Prop. 3.3]{instantons}, $A(t)$ converges in $C_{\text{loc}}^\infty$ to an ASD instanton $A_\infty.$ \par
    Next we show global convergence of $A(t)$ to $A_\infty$ in $L^q.$ It suffices to show that
    \begin{align*}
        \int_N^\infty \|D^\ast F\|_q < \infty.
    \end{align*}By H{\"o}lder's inequality and the first part of the proof,
    \begin{align*}
        \int_N^\infty \|D^\ast F\|_q &\leq \int_N^\infty \|D^\ast F\|_2^{\frac{2}{q}}\|D^\ast F\|_\infty^{1-\frac{2}{q}} \\
        &\lesssim \sum_{n = N}^\infty n^{-\frac{2}{p}(1-\frac{2}{q})}\int_n^{n+1} \|D^\ast F\|_2^{\frac{2}{q}} \\
        &\lesssim \sum_{n = N}^\infty n^{-\frac{2}{p}(1-\frac{2}{q})}\left(\int_n^{n+1} \|D^\ast F\|_2^{2}\right)^{\frac{1}{q}}
    \end{align*} By (\ref{energyineq:ineq}), Lemma \ref{lemma:Lpcontrolafter}, and Lemma \ref{lemma:improveddecay},
    \begin{align*}
        \int_n^{n+1} \|D^\ast F\|^2 &\leq \|F^+(n)\|^2 \\
        &\leq (\|F^+(n)\|_p^{\frac{p}{2}}\|F^+(n)\|_\infty^{1-\frac{p}{2}})^2 \\
        &\lesssim \|F^+(n)\|_\infty^{2-p} \\
        &\lesssim n^{2-\frac{4}{p}}.
    \end{align*} Therefore
    \begin{align*}
        \int_N^\infty \|D^\ast F\|_q \lesssim \sum_{n = N}^\infty n^{\frac{2}{q}-\frac{2}{p}},
    \end{align*} which is finite if $q > \frac{2p}{2-p},$ as desired. \par
    Finally we note that $F(t)$ converges in $L^2$ to $F_{A_\infty}.$ Fix $x \in M.$ Since
    \begin{align*}
        \int_{M \setminus B_R(x)} |F(t_2) - F(t_1)|^2 \leq 2 \int_{M \setminus B_R(x)} |F(t_2)|^2 + |F(t_1)|^2,
    \end{align*} and since $F(t)$ converges to $F_{A_\infty}$ in $C_{\text{loc}}^\infty$ by the first part of the proof, it suffices to show that for any $\varepsilon > 0,$ there exists $R, T > 0,$ such that for $t \geq T$
    \begin{align*}
        \int_{M \setminus B_R(x)} |F(t)|^2 \leq \varepsilon.
    \end{align*} This follows from (\ref{F+L1}) and \cite[Theorem 2.5]{instantons}, where in the statement we replace $B_{\frac{R}{N}}(x)$ with $M \setminus B_{NR}(x)$ and $B_R(x)$ with $M \setminus B_R(x),$ and in the proof we replace the special cutoff function $\beta_{N, R}$ with $1 - \beta_{N, NR}.$ Since the metric on $M$ approaches Euclidean on the end sufficiently fast, the equivalent of \cite[Lemma 2.4]{instantons} holds for $1 - \beta_{N, NR},$ so the proof of \cite[Theorem 2.5]{instantons} goes through.
\end{proof}

\begin{proof}[Proof of Theorem \ref{thm:main}$c$]
%The decay of the $L^\infty$ and $L^p$ norms of $F^+$ was established in \textsection 4, and the convergence of $A(t)$ is the content of the preceding lemma.
The convergence statement is contained in the previous theorem. By Lemma \ref{lemma:energyidentity}, we have $\kappa(A(t)) \equiv \kappa(A_0)$ for all $t < \infty.$ Since the curvature converges in $L^2$ as $t \to \infty,$ we also have $\lim_{t \to \infty} \kappa(A(t)) = \kappa(A_\infty).$ We conclude that $\kappa(A_0) = \kappa(A_\infty).$
\end{proof}

\bibliographystyle{amsinitial}
\bibliography{biblio}

\end{document}